\documentclass{amsart}

\newtheorem{thm}{Theorem}[section]
\newtheorem{theorem}[thm]{Theorem}

\newtheorem{corollary}[thm]{Corollary}

\newtheorem{lemma}[thm]{Lemma}

\newtheorem{proposition}[thm]{Proposition}

\theoremstyle{definition}

\newtheorem{definition}[thm]{Definition}

\theoremstyle{remark}
\newtheorem{remark}{Remark}[section]

\newtheorem{question}{Question}

\newtheorem{Fact}{Fact}

\def\val{\operatorname{val}}

\usepackage{amssymb}
\usepackage{amsrefs}

\DefineSimpleKey{bib}{primaryclass}{}
\DefineSimpleKey{bib}{archiveprefix}{}

\BibSpec{arXiv}{%
  +{}{\PrintAuthors}{author}
  +{,}{ \textit}{title}
  +{}{ \parenthesize}{year}
  +{,}{ arXiv }{eprint}
  +{,}{ primary class }{primaryclass}
}

\def\dom{\operatorname{dom}}

\begin{document}
\title{Not OCA and products of Fr\'echet spaces}
\author{Alan Dow}
\address{Department of Mathematics and Statistics, University
of North Carolina at Charlotte, Charlotte, NC, USA}

\begin{abstract}
We continue the investigation of the question of whether the product of two
countable Fr\'echet spaces must be M-separable. 
We are especially interested in this question in the presence
of Martin's Axiom.  The question has been shown to be
independent of Martin's Axiom but only in models in
which $\mathfrak c\leq\omega_2$.  In fact, OCA
implies an affirmative answer.
\end{abstract}


\maketitle

\section{Introduction}
A space is Fr\'echet (or Fr\'echet-Urysohn) if 
it satisfies that if a point $x$ is in the closure
of a set $A$, there is a standard $\omega$-sequence
of points from $A$ converging to $x$. 
A space $X$ is M-separable, also known as
selectively separable, if for every countable family
 $\{ D_n : n\in \omega \}$ of dense subsets, there is 
 selection $\{ H_n\in [D_n]^{<\aleph_0} : n\in \omega\}$ 
 satisfying that $H = \bigcup_n H_n$ is dense.
 Every
separable Fr\'echet space is M-separable \cite{BaDo1}.
A product of two separable Fr\'echet spaces need
not be Fr\'echet, but might the product still be
M-separable?  

There are two main results of this paper.
The first is   that in models
of Martin's Axiom in which there are special  
$(\mathfrak c, \mathfrak c)$-gaps, there will
be pairs of countable Fr\'echet spaces with a product
that is not M-separable. The second is that 
in standard models of $\operatorname{MA}(\sigma\mbox{-linked})$,
it will hold that the product of any two countable
Fr\'echet spaces will be M-separable. 
In both results there is no (new) restriction on the size
of the continuum.  
 $\operatorname{MA}(\sigma\mbox{-linked})$ is 
 the statement (see \cite{AvTo11}) that Martin's Axiom holds
 for ccc posets that can be expressed as a
 countable union of linked subsets.

Any countable space with $\pi$-weight less than $\mathfrak d$
is M-separable and, in the Cohen model every countable
Fr\'echet space has $\pi$-weight at most $\aleph_1$
 \cite{BaDo2}.  Therefore we are more interested in
 the question in models in which there are countable
 Fr\'echet spaces with $\pi$-weight at least  $\mathfrak d$. 
 In fact in this paper we focus on models in which
  $\mathfrak b = \mathfrak c$. It is interesting
  that it was shown in \cite{RZbd} that $\mathfrak b = \mathfrak d$
  implies there are countable M-separable spaces whose product
  is not M-separable, but the status of this statement  in ZFC
  is very much open. The cardinals $\mathfrak p, \mathfrak b$,
  and $\mathfrak d$ are the usual cardinal invariants corresponding
  to mod finite orderings on subsets of $\omega$
  known as the pseudointersection number,
   the bounding number, and the dominating number.
 
 Back to the product of countable Fr\'echet spaces in models
 of $\mathfrak b = \mathfrak c$. The known results seem 
 to point to a close connection to the open coloring axiom
 and gaps. It was shown in \cite{BFZ} that in a model 
 of Martin's Axiom plus $\mathfrak c=\omega_2$ in which
 there was a strong failure of OCA,
  this strong failure of OCA was crucial to the
  construction  of two
 countable Fr\'echet spaces whose product was
 not M-separable. Improving on the PFA result
 in \cite{BaDo2}, it was shown in \cite{MA3} that the
 version of 
 OCA from \cite{StevoPPT}, which was shown  to imply $\mathfrak b=\omega_2$,
 implies that the product of two countable Fr\'echet spaces
 is necessarily M-separable.  The key step in this 
 proof was that this version of OCA  implies that most unseparated pairs
 of orthogonal
 ideals on $\omega$ will contain a Luzin gap (see Definition \ref{luzin}).
 Two families of subsets of a countable set are said to be orthogonal
 if each member of the first family has finite intersection with
 each member of the second.
 The other well-known version of OCA, denoted OCA$_{[\operatorname{ARS}]}$ in 
 \cite{Justin}, is from \cite{ARS}.
 On the other hand, 
 adding to the seeming OCA connection, it was 
 also shown in \cite{MA3} that 
 Martin's Axiom plus not CH implies there are three
 countable Fr\'echet spaces whose product is not M-separable
 \textbf{because}, by Avil\'es-Todor\v{c}evi\'{c} \cite{AvTo11}, in
 such models there are no three dimensional analogues
 of the above mentioned   Luzin gaps.

 We note that  a space $X$ is M-separable 
if every countable filter base
$\mathcal D$ consisting of dense subsets of a space $X$
has a pseudointersection that is dense.  
 Since M stands for Menger, one could define the Menger
 degree of a space $X$, $\mathfrak p$-M$(X)$, to 
 be the minimum cardinality of a filter base of
 dense subsets of $X$ that has no dense pseudointersection.
  A natural family to consider is those countable
  Fr\'echet spaces
 with Menger degree equal to the pseudointersection number
  $\mathfrak p$. We did begin work on this paper 
  by considering whether this smaller family
  of spaces may have better behavior in products
  but could find no results. We leave this remark here
as a simple suggestion for further research.

Motivated by the  paper \cite{Dow2025},
we had hoped to completely solve this  question in this paper,
but it remains open.

\begin{question}  Does $\mbox{MA} + \mathfrak c>\omega_2$ imply
there are two countable Fr\'echet spaces whose product
is not M-separable?
\end{question}

\section{A few combinatorial tools}

 The following observation, a strengthening of
  Arhangel'skii's $\alpha_1$-property for first countable
  spaces, has proven useful in a number of papers. This statement and proof
  is taken right from \cite{MA3} and is included
  for completeness.

  \begin{proposition} Let $\mathcal I$ be a family
  of sequences\label{alpha1} in a countable space $X$ all converging
  to a single point $x$ that has countable character.
  If $\mathcal I$ has cardinality less than $\mathfrak b$,
   then there is a single sequence $S$ converging to
    $x$ that mod finite contains every member of $\mathcal I$.
  \end{proposition}

\begin{proof} Fix a  descending neighborhood
basis, $\{U_n : n \in\omega\}$, for $x$
with $U_0=X$. For each
 $n\in\omega$, let $X_n = U_n\setminus U_{n+1}$.
 There is nothing to prove if $x$ has a neighborhood
 that is simply a converging sequence, so
  we may assume that each $X_n$ is infinite.
For each $n\in\omega$,
 choose an  enumeration, $\{ x(n,m) : m\in \omega\}$
 of $X_n$. For each $I\in\mathcal I$, there is a
 function $f_I\in \omega^\omega$ satisfying
 that $I \subset \bigcup_n   \{ x(n,m) : m<f_I(n)\}$. 
 Therefore, if $|\mathcal I|<\mathfrak b$, we may
 choose a function $f\in\omega^\omega$ so that
  $f$ is eventually larger than each $f_I$. 
  It is easy to check that $S =\bigcup_n
   \{ x(n,m) : m< f(n)\}$ is a sequence
   that converges to $x$ 
   and which satisfies that $I\setminus S$
   is finite for all $I\in \mathcal I$.
\end{proof}

For a set $A$ in a space $X$, we let
 $A^{(1)} $ be the set of points $x$ of $X$ for which
 there is a countable, possibly constant,
 sequence from $A$ converging to $x$.

\begin{proposition}[\cite{Dow2014}] If a space $X$ has\label{frechet} character less
than $\mathfrak b$, then for every set $A\subset X$,
  the set $(A^{(1)})^{(1)} = A^{(1)}$.
\end{proposition}

The next two items are taken from \cite{StevoPPT}*{\S 8}
and \cite{Farah2000}*{Theorem 2.2.1}

\begin{definition} A family 
 $\{ (I_\alpha,J_\alpha) : \alpha<\omega_1\}$  
 is a Luzin gap\label{luzin} 
 if  $\bigcup\{ I_\alpha\cup J_\alpha : \alpha\in\omega_1\}$ is countable,
 , for each $\alpha\neq \beta$, $I_\alpha\cap J_\alpha$ is empty,
 $I_\alpha\cap J_\beta$ is finite, and
  $(I_\alpha\cap J_\beta)\cup (I_\beta\cap J_\alpha)$ is not empty.
\end{definition}

\begin{proposition}  If $\{ (I_\alpha,J_\alpha) : \alpha<\omega_1\}$ is a Luzin gap,
then\label{nosplit} the family $\{ I_\alpha : \alpha\in\omega_1\}$ can not be mod
finite separated, or \textit{split}, from the family 
 $\{ J_\alpha : \alpha\in\omega_1\}$. 
\end{proposition}

For completeness, we include   the easy proof.

\begin{proof}  Fix an enumeration $e:\omega\rightarrow 
\bigcup\{ I_\alpha\cup J_\alpha : \alpha\in\omega_1\}$.
Suppose that  $A$ is   a set satisfying
that each of  $I_\alpha\setminus A$ and
 $J_\alpha\cap A$ are  finite  
for all $\alpha\in\omega_1$.
Choose a finite subset $F$ of $\omega$ so that there is an 
uncountable subset $\Gamma$ of $\omega_1$ satisfying that,
for all $\alpha\in \Gamma$, 
 $I_\alpha\setminus A\subset e(F)$ 
 and $J_\alpha\cap A\subset e(F)$.
If necessary, shrink $\Gamma$ further so that, 
 for all $\alpha,\beta\in \Gamma$,
 $I_\alpha\cap e(F)=I_\beta\cap e(F)$ and
  $J_\alpha\cap e(F) = J_\beta\cap e(F)$. 
  Note that for $\alpha\neq\beta\in \Gamma$, 
   $I_\alpha \setminus e(F)\subset A$
   and $J_\beta \setminus e(F)$ is disjoint from $A$.
   Since, in addition, $I_\alpha\cap e(F)$ is disjoint
   from $J_\alpha \cap e(F) = J_\beta\cap e(F)$, 
   this contradicts that the family was Luzin.
\end{proof}

\section{MA, a failure of OCA and two Fr\'echet spaces}

\begin{definition} Say that two ideals $\mathcal I_1, \mathcal I_2$
form a tight  $\omega^\omega$-gap if 
\begin{enumerate}
\item every member of $\mathcal I_1\cup \mathcal I_2$ is a subset
of $\omega\times\omega$,
\item for each $I\in \mathcal I_1\cup \mathcal I_2$, $I$ is a subset
of $f^{\downarrow}  = \{ (n,m) : m < f(n)\}$    for some
 $f\in \omega^\omega$, 
\item $\mathcal I_1$ and $\mathcal I_2$ are orthogonal,
\item  for each $f\in \omega^\omega$, there are $a_f\in\mathcal I_1$
and $b_f\in \mathcal I_2$ such that $f^{\downarrow}= a_f\cup b_f$,
\item for any $X\subset \omega\times\omega$ such that
 $\{ n\in\omega : X\cap (\{n\}\times\omega)\ \mbox{is infinite}\}$
 is infinite, 
there is an $f\in \omega^\omega$ such that each of
 $X\cap a_f$ and $X\cap b_f$ are infinite.
\end{enumerate}
\end{definition}

\begin{remark}  CH implies there are tight $\omega^\omega$-gaps.
Todor\v{c}evi\'{c} \cite{StevoPPT} showed that OCA implies
there is no tight $\omega^\omega$-gap.
It is proven in \cite{Dow2025} that it is consistent with
 Martin's Axiom and $\mathfrak c$ arbitrarily large that there is
 no tight $\omega^\omega$-gap.  
\end{remark}

 The proof of the following Lemma 
 is technical and is simply following the methods of Laver \cite{Laver},
  see also Rabus \cite{Rabus}*{Theorem 1}
  and Scheepers \cite{Scheepers}, showing that $(\omega_1,\omega_1)$-gaps
  that are added generically, can be split by a ccc poset.
  We postpone the proof to the last section.

\begin{lemma} For any\label{tight}
cardinal $\kappa >\omega_1$ such that $\kappa^{<\kappa}=\kappa$,
there is a ccc poset $P$ such that in the forcing extension by $P$,
Martin's Axiom holds, $\mathfrak c=\kappa$, and 
there is a tight $\omega^\omega$-gap. 
\end{lemma}

\begin{theorem} If there is a tight $\omega^\omega$-gap and $\mathfrak b = \mathfrak c$,
 then there are two countable Fr\'echet spaces whose product is not M-separable.
\end{theorem}

\begin{proof} Assume that $\mathfrak b = \mathfrak c$ and that
$\mathcal I_1,\mathcal I_2$ form a tight $\omega^\omega$-gap. 
Let $\{ f_\alpha : \alpha \in \mathfrak c\}\subset \omega^\omega$ 
be   a standard scale in the sense that, each $f_\alpha$ is a strictly increasing
function, $f_\alpha<^*f_\beta$ for any $\alpha < \beta < \mathfrak c$,
and for all $f\in \omega^\omega$ there is an $\alpha<\mathfrak c$
such that $f\leq f_\alpha$. 

For each $\alpha<\mathfrak c$, let $a_\alpha\in \mathfrak I_1$
and $b_\alpha\in \mathfrak I_2$ be disjoint sets such that
 $a_\alpha\cup b_\alpha = f_\alpha^{\downarrow}$. 
 \medskip

We start our construction as in   \cite{BFZ}, and similar to \cite{MA3}.
 Let $\tau_0=\sigma_0$ be any countable clopen base for a topology
 on $\omega$ that is homeomorphic to the rationals. 
 Fix a partition $\{ E_n : n\in\omega\}$ of $\omega$ so that
 each $E_n$ is $\tau_0$ dense. Now choose any bijection $\rho$
 on $\omega$ satisfying that, for each $n\in\omega$,
   the graph of $\rho\restriction E_n$, namely $D_n =
   \{ (k, \rho(k)) : k\in E_n\}$
   is a dense subset of $\omega\times\omega$ with respect to the 
   product topology $\tau_0\times \sigma_0$. 
   Observe that $\rho[E_n]$ is dense with respect to $\sigma_0$.
   Let $D =\bigcup_n E_n = \{ (k, \rho(k)) : k\in \omega\}$,
   i.e. $D$ is the graph of $\rho$.   We will let
   $\pi_1$ denote the first coordinate projection on $\omega\times\omega$
   and $\pi_2$ the second coordinate projection.
   
   Choose any countable elementary submodel $M_0$ of 
    $H(\mathfrak c)$ such that each of 
     $\{    \{E_n : n\in\omega\}, \rho, \tau_0\}\in M_0$.
     This is simply a convenient way of choosing a good
     starting family of converging sequences from
     each of $\tau_0$ and $\sigma_0$.  Let
     $\mathcal I_0$ be all sets $I\in M_0$
such     that, for some $n\in\omega$, $I$ is a subset
     of $E_n$ and $\tau_0$-converges.
     Similarly let $\mathcal J_0$ be all sets
     $J\in M_0$ that, for some $n\in \omega$,
      $J$ is a subset of $\rho[E_n]$ and $\sigma_0$-converges.

      Let us note here that if $\tau'\supset \tau_0$
      and $
      \sigma'\supset \sigma_0$ are larger
      bases for topologies that preserve that 
      $\mathcal I_0$ and $\mathcal J_0$ respectively remain
      converging, then, for each $n\in\omega$,
         $D_n = \rho\restriction E_n$ is dense in the product
         topology. To see this let $U\in \tau'$ and $W\in \sigma'$.
         Choose any $m\in U$ and $k\in W$. In $M_0$ choose an infinite 
         sequence
         $S\subset D_n$ such that $S$ converges to $(m,k)$.
         Therefore $S = \{ (i,\rho(i)) : i\in I\}$ for some
         $I\in M_0$. Note that  $I\in \mathcal I_0$ and converges
         to $m$. Similarly,     $\rho[I] = J$ is an element of $\mathcal J_0$
           and $J$ converges to $k$. By assumption $I$ is almost
           contained in $U$ and $J$ is almost contained in $W$. 
           Of course this implies that $U\times W$ almost contains
           $S\subset D_n$.\medskip

Choose a bijection $\psi : \omega\times \omega \rightarrow \omega$
such that $\psi(\{n\}\times \omega) = E_n$. 
For each $\alpha<\mathfrak c$, let $A_\alpha = \rho(\psi(a_\alpha))$
and $B_\alpha = \rho(\psi(b_\alpha))$. Now we have 
that the ideals generated by $\{ \rho[E_n] : n\in \omega\}$,
 $\{ A_\alpha : \alpha < \mathfrak c\}$, and 
  $\{B_\alpha : \alpha < \mathfrak c\}$ are orthogonal ideals
  on $D$ whose union is dense in $D$. If $X$ is a subset
  of $D$ that meets infinitely many of the elements
  of $\{ \rho[E_n] : n\in\omega\}$ in an infinite set,
   then there is an $\alpha<\mathfrak c$ such that
    $X\cap A_\alpha$ and $X\cap B_\alpha$ are both infinite.

   We will recursively construct increasing chains, 
    $\{\tau_\alpha : \alpha < \mathfrak c\}$ and
     $\{ \sigma_\alpha : \alpha < \mathfrak c\}$ 
     of clopen bases of cardinality less than $\mathfrak c$
     for topologies on $\omega$. We will also, simultaneously
     choose increasing chains, $\{
     \mathcal I_\alpha : \alpha <\mathfrak c\}$
     and   $\{\mathcal J_\alpha : \alpha < \mathfrak c\}$,
      of sequences that must converge in $\tau_\beta$, 
      respectively $\sigma_\beta$, for all $\beta < \mathfrak c$.
      Naturally the purpose of choosing these chains of
      sets of converging sequences is to ensure that
      each of $(\omega,\tau_{\mathfrak c})$
      and $(\omega,\sigma_{\mathfrak c})$ are Fr\'echet as
      witnessed by $\mathcal I_{\mathfrak c}$ and
       $\mathcal J_{\mathfrak c}$ respectively.

       As mentioned above, with these inductive assumptions,
       we will have ensured that each member
       of the sequence $\{ D_n : n\in\omega\}$
       remains 
   $\tau_\alpha\times \sigma_\alpha$-dense for
        every $\alpha$.
        The next goal is to ensure that  
        if $H\subset D$ satisfies
        that $H\cap D_n $ is finite for all
         $n$, then $H$ is closed and discrete.
         This of course ensures that the product
         is not M-separable. 
    Here is the plan for ensuring that such an
     $H$ is closed and discrete. Notice that
     there will be an $\alpha<\mathfrak c$
     such that $H\subset^* A_\alpha\cup B_\alpha$.
     We will ensure that the first coordinate projection,
     $\pi_1(H\cap B_\alpha)$, is 
     closed and discrete in $(\omega,\tau_{\alpha+1})$
     and that the second coordinate projection,
     $\pi_2[H\cap A_\alpha]$, is closed
     and discrete in $(\omega,\sigma_{\alpha+1})$.
     To ensure this is possible, we will necessarily 
     also have 
     the inductive hypotheses that
    if $I\in\mathcal I_\alpha$ is almost
    disjoint from each $E_n$, then
     $\rho\restriction I $ is contained in
     some $A_\gamma$.
     Similarly, if $J\in \mathcal J_\alpha$
     and is almost disjoint from 
     each $\rho[E_n]$, then $\rho^{-1}\restriction J
      = \{ (\rho^{-1}(j) , j) : j\in J\}$ is 
      almost contained in some  $B_\gamma$.
      These conditions are vacuously true for $\mathcal I_0$
      and $\mathcal J_0$.
     
Fix an enumeration $\{ X_\xi : \xi < \mathfrak c,\ \xi\ \mbox{a limit} \}$ of the
infinite subsets of $\omega$. Let $0<\lambda < \mathfrak c$ and
assume that we have constructed the following increasing sets
$\{ \tau_\alpha : \alpha <\lambda\}$, 
 $\{ \sigma_\alpha : \alpha < \lambda \}$,
 $\{\mathcal I_\alpha : \alpha < \lambda\}$,
 and 
  $\{\mathcal I_\alpha : \alpha < \lambda\}$ 
satisfying the following
 inductive assumptions for all $\beta < \alpha <\lambda$:
 \begin{enumerate}
 \item every $I\in \mathcal I_\alpha$ is a $\tau_\alpha$-converging,
 \item every $J\in \mathcal J_\alpha$ is a $\sigma_\alpha$-converging,
 \item for each $I\in \mathcal I_\alpha$, the graph
 $\rho\restriction I$  is mod finite contained in $D_n\cup A_\gamma$
 for some $n\in\omega$ and $\gamma <\mathfrak c$,
 \item for each $J\in \mathcal J_\alpha$, the set
  $\rho^{-1}\restriction J = \{ (\rho^{-1}(j),j) : j\in J\}$ is
  mod finite contained in $D_n\cup B_\gamma$ for some
  $n\in\omega$ and $\gamma < \mathfrak c$,
  \item the set $\pi_1[B_\beta]  $
    is closed and discrete with respect to $\tau_\alpha$,
    \item the set $\pi_2[A_\beta]$ 
    is closed and discrete with respect to $\sigma_\alpha$,
\item if $\beta$ is  a limit 
and $\beta+m  < \alpha$,
 then if $m$ is a  $\tau_\alpha$-limit point of $X_\beta$, 
           there is an $I\in \mathcal I_\alpha$ converging to $m$
           such that              $I\subset X_\beta$,
\item if $\beta$ is a limit and $\beta+k\leq \alpha$,
 then if  $k$ a $ {\sigma_\alpha}$-limit point of $ X_\beta$,
           there is a $J\in \mathcal J_\alpha$ converging to $k$
           such that              $J\subset X_\beta$.
 \end{enumerate}

 If $\lambda$ is a limit ordinal, then the inductive hypotheses
 are satisfied by simply taking unions: 
 $\tau_\lambda = \bigcup_{\alpha<\lambda}\tau_\alpha$,
  $\sigma_\lambda = \bigcup_{\alpha<\lambda}\sigma_\alpha$,
  $\mathcal I_\lambda = \bigcup \{ \mathcal I_\alpha: \alpha < \lambda\}$,
and   $\mathcal J_\lambda = \bigcup \{ \mathcal J_\alpha : \alpha<\lambda\}$.

Now suppose that $\lambda = \alpha{+}1$ and, if $\omega\leq \alpha$,
let $\beta$ be the largest limit below $\lambda$ and let
 $\beta + \bar m +1 = \lambda$.  
 There are two tasks for each 
of $\tau_{\lambda}$ and $\sigma_{\lambda}$ to deal with $X_\beta$,
 $A_\alpha$ and $B_\alpha$ as in items (5)-(8).  
 These are done independently, but symmetrically, for 
  $\tau_\lambda, \mathcal I_\lambda$ and $\sigma_\lambda,\mathcal J_\lambda$,
  so we just provide the construction for $\tau_\lambda$ and $\mathcal I_\lambda$.
 
 Let us
 first consider the closure of $X_\beta$ with respect to
 $\tau_\alpha$. For each $m\in\omega$ for which there
 is a sequence $I\subset X_\beta$ that converges to
  $m$ and satisfies that $\rho\restriction I \subset E_n\cup A_\gamma$
  for some $n\in\omega$ and $\gamma <\mathfrak c$, ensure
  there is such an $I\in \mathcal I_\lambda$. 
Let $ X_\beta^{(1)_\lambda}$ denote the set $ X_\beta$ together
with all the points that are   Fr\'echet limits with respect
to $\mathcal I_\lambda$. 
Naturally this step is only required at stage $\lambda = \beta+1$.
Note that it follows from 
Proposition \ref{frechet} that $X_\beta^{(1)_\lambda}$ is 
almost disjoint from every $I\in\mathcal I_\lambda$
such that $I$ converges to a point not in $X_\beta^{(1)_\lambda}$.

Apply Lemma \ref{alpha1} to choose, for each $m\in \omega$
a sequence $S_m$ that $\tau_0$ converges to $m$ and satisfies
that $I\subset^* S_m$ for all $I\in\mathcal I_\lambda$ that 
converge to $m$. Clearly the family $\{ S_m : m\in\omega\}$
is almost disjoint, and so by removing a finite set from
each, we will assume they are pairwise disjoint. 
A second reduction is that we can replace each $S_m$
by  $S_m\setminus \pi_1[B_\lambda]$ since, by our inductive
assumptions, each $I\in \mathcal I_\lambda$ is almost
disjoint from $\pi_1[B_\lambda]$. Our third, and final reduction,
is that for each $m\notin X_\beta^{(1)_\lambda}$,   
we can assume that $S_m\cap X_\beta^{(1)_\lambda}$ is empty, but
this needs a proof since we can not apply Proposition  \ref{frechet} directly
because of the new restriction that we must respect the $\omega^\omega$-gap.

Assume that there is a sequence $\{ s_n: n\in\omega\}
\subset X_{\beta}^{(1)_\lambda}$ that converges to $m$.
We prove that $m$ is also in $X_\beta^{(1)_\lambda}$.
For
each $n\in\omega$, fix a sequence $I_n$ that converges to
 $n$ and such that $I_n\subset X_\beta$ and, by definition
 of $X_\beta^{(1)_\lambda}$, 
 either $I_n\subset E_{k_n}$  (case 1) or   $I_n\subset \pi_1[A_{\gamma_n}]$ (case 2). 
 By passing to a subsequence of $\{ s_n : n\in\omega\}$ we may
 assume that either,   for all $n$, $I_n\subset E_{k_n}$ (case 1) or
   for all $n$, $I_n\subset \pi_1[B_{\gamma_n}]$ for some $\gamma_n <\mathfrak c$
   (case 2).

Since it is easier, we complete the proof for case 2 first.
Choose any $\gamma <\mathfrak c$ so that $\gamma_n <\gamma$ for all $n$.
 By removing  finite subset from each $I_n$ we may assume that
  $\bigcup_n I_n\subset \pi_1[A_\gamma]$. Now the character of $m$
  in the subspace $\{m\}\cup \{s_n :n\in\omega\}\cup \{ I_n : n\in\omega\}$
  is less than $\mathfrak b$ and if we set $A = \bigcup_n I_n$, we
  can apply Proposition \ref{frechet} to conclude there is a 
  sequence $I\subset \bigcup_n I_n$ that converges to $m$. This
  implies that $m$ is in $X_\beta^{(1)_\lambda}$. 

Now we deal with case 1.
If there is a $k$ so that $k_n = k$ for infinitely many $n$,
 then clearly, by Proposition \ref{frechet}, there is a sequence
contained in $X_\beta\cap E_{k}$ that converges to $m$, showing
that $m\in X_{\beta}^{(1)_\lambda}$. 
Finally, again by passing to a subsequence, we may assume
that $\{ k_n : n\in\omega\}$ is strictly increasing.
Let $\{ U_\xi : \xi < \lambda\}$ enumerate the neighborhood
base at $m$ with respect to the topology $\tau_\alpha$. 
For each $\xi<\lambda$, there is a function $h_\xi\in\omega^\omega$
so that $I_n\setminus h_\xi(n)\subset U_\xi$ for all but finitely
many $n\in\omega$. Since $\lambda<\mathfrak b$, there is a
$\gamma<\mathfrak c$ such that, for all $\xi <\lambda$, 
$I_n\setminus f_\gamma(k_n)$ is a subset of $U_\xi$
for all but finitely many $n\in\omega$. Let $X =\bigcup_n 
\rho[I_n\setminus f_\gamma(k_n)]$
and note  $X$ is a subset that meets infinitely many of
the elements of $\{ \rho[E_k] : k\in\omega\}$ in an infinite
set.  By the assumption on the $\omega^\omega$-gap, 
 there is a $\delta<\mathfrak c$ such that $X\cap A_\delta$
 is infinite. Since $A_\delta\cap \rho[I_n]$ is finite
 for every $n$, it follows that $X\cap A_\delta$ is 
 mod finite contained in $\rho[U_\xi]$ for every $\xi<\lambda$.
 Equivalently, $\pi_1[X\cap A_\delta]\subset X_\beta$ is a sequence 
 that converges to $m$ with respect to $\tau_\alpha$, showing
 that $m\in X_\beta^{(1)_\lambda}$.
\medskip

Now we construct countably many new clopen sets to
add to $\tau_\lambda\supset \tau_\alpha$ by defining
a function $g:\omega\mapsto \omega$ and adding
$g^{-1}(k)$ to $\tau_\lambda$ for each $k\in\omega$.
Then let $\tau_\lambda$ be closed under finite intersections.

 Let $g_0$ be any 1-to-1 function from $\bar m\cup \pi_1[B_\alpha]$
 into $\omega$. For convenience choose $g_0(\bar m)=0$.
 We define $g$  as $\bigcup g_n$ where for each
 $n\in\omega$, $\dom(g_n)$ equals $\{\bar m\}\cup 
 \pi_1[B_\alpha]\cup n \cup \bigcup\{S_m: m<n\} $.
 Note that with this assumption, $\dom(g_n)$ is almost disjoint
 from $S_\ell$ for all $\ell\geq n$. Two additional
 inductive assumption are that for each $m,j\in \dom(g_n)$,
 \begin{enumerate}
  \item  if $m<n$, then $g(s)=g(m)$ for all but finitely many $i\in S_m$,
\item     $j\in X_\beta^{(1)_\lambda}$  
then $g_n(j)\neq 0=g_n(\bar m)$.  
\end{enumerate}
The first inductive assumption on $g$ ensures that every member
of $\mathcal I_\lambda$ will be $\tau_\lambda$-converging,
and the second ensures that $g^{-1}(0)$  
is a $\tau_\lambda$-neighborhood of $\bar m$ that is disjoint from
  $X_\beta $.

Our definition of $g_0$ ensures that $\pi_1[B_\alpha]$ is closed
and discrete.    Assume then that
$g_n$ has been defined and note that the inductive assumptions
ensure that the range, $R_n$, of $g_n\restriction (\dom(g_n)\setminus \pi_1[B_\alpha])$ 
is finite. If $n$ is not in $\dom(g_n)$, then define $g_{n+1}(n)$ to be
any value not in $R_n$.  So long as $g_{n+1}(n)\neq 0$,
then simply    
define $g_{n+1}(i)=g_{n+1}(n)$ for all $i\in S_n\setminus
\dom(g_n)$.  
This choice of $g_{n+1}\restriction S_n$ preserves both the inductive
hypotheses.  
If $g_{n+1}(n)=0$, then it is because $g_n(n)=0$
and by the induction hypothesis, $n\notin X_\beta^{(1)_\lambda}$
and $S_n$ is also disjoint from $X_\beta^{(1)_\lambda}$.
For these reasons,  we may again define $g_{n+1}(i)=0$
for all $i\in S_{n}\setminus \dom(g_n)$ and preserve
the induction hypotheses.
\medskip

Let us verify that $(\omega, \tau_{\mathfrak c})$ is Fr\'echet,
where $\tau_{\mathfrak c} = \bigcup_{\alpha<\mathfrak c} \tau_\alpha$.
Consider any $\bar m\in\omega$ and subset $X_\beta\subset\omega$
for some limit $\beta\in\mathfrak c$. If $\bar m$ is 
in the set $X_\beta^{(1)_{\beta+1}}$, then there is a sequence
 $I\in \mathcal I_{\beta+1}$ that is contained in $X_\beta$
 and converges to $\bar m$. Otherwise,  by induction
 hypothesis (7), $\bar m$ is not in the $\tau_{\beta+\omega}$-closure
 of $X_\beta$. 
\end{proof}
 
\section{Products which are M-separable}

In this section we prove that in standard models of weak forms of Martin's Axiom
 products of two countable Fr\'echet spaces are M-separable. 
  It is known that this holds in models of OCA but we are interested
  in models in which $\mathfrak c$ is larger than $\omega_2$. It is
  shown in \cite{BaDo2} that this also holds
  in all standard Cohen real forcing extensions. We do not know
  if this can hold in models of Martin's Axiom with $\mathfrak c>\omega_2$,
   so we make do with 
   Martin's Axiom for $\sigma$-linked posets (\cite{AvTo11})
   as was used for products of three Fr\'echet spaces in \cite{MA3}.

\begin{theorem}  Let $\kappa>\omega_1$ satisfy that $\kappa^{<\kappa}=\kappa$.
Then if   $P_\kappa$ is\label{second}  is the standard finite
support iteration of length $\kappa$  consisting of factors
that are names of  $\sigma$-linked posets, then in the forcing
extension any product of two countable Fr\'echet spaces is M-separable.
\end{theorem}

The following is a direct consequence of
\cite{StevoPPT}*{Theorem 4.4}, see also 
\cite{Justin}*{Lemma 1}.  It is the key result behind the proof in \cite{StevoPPT}*{Theorem 8.0}
that PFA implies OCA.

\begin{lemma}[CH] If $X$ is a separable metric space and $G\subset X^2\setminus\Delta_X$ is a 
symmetric open\label{preOCA} relation on $X$, 
then either there is a countable cover, $\mathcal Y$, of $X$ by sets
 $Y\in\mathcal Y$ satisfying that $Y^2\cap G$ is empty, or
 the poset $Q=\{ F\in [X]^{<\aleph_0} : F^2\setminus \Delta_F \subset G\}$ is
 ccc when ordered by reverse inclusion.
\end{lemma}

Therefore, using the standard countably closed collapsing trick
 (see again \cite{StevoPPT}*{Theorem 8.0}) and the fact that the iteration
 of a countably closed poset and a ccc poset is proper we
 have the same result in a more convenient form.

\begin{corollary}
 If $X$ is a separable metric space and $G\subset X^2\setminus\Delta_X$ is a 
symmetric open\label{properOCA} relation on $X$ for which
there is no countable cover, $\mathcal Y$, of $X$ by sets
 $Y\in\mathcal Y$ satisfying that $Y^2\cap G$ is empty,
  then there is   a proper poset $P$ that forces there
  to be an uncountable set $Z\subset X$ satisfying
  that $Z^2\setminus \Delta_Z$ is a subset of $G$.
\end{corollary}

Using this result and the method from \cite{MA3} we have this next
technical Lemma 
concerning products of Fr\'echet spaces.
We can loosely view it as forcing the product to be M-separable.

\begin{theorem} Let $(\omega,\tau)$ and $(\omega,\sigma)$ be Fr\'echet spaces. 
Assume that $\{ D_n : n\in\omega\}$ are dense subsets of $\omega\times\omega$
with respect\label{technical} to the product topology. Let $(x,y)\in \omega\times\omega$
be arbitrary and let $\mathcal I, \mathcal J$ be the family of all
sequences that $\tau$-converge, respectively $\sigma$-converge, to
 $x$ and $y$ respectively. 

If $Q$ is any $\sigma$-linked poset that adds a dominating real
 $f$, then in the forcing extension by $Q$, if $\hat\tau$ and 
 $\hat \sigma$ are topologies extending $\tau$ and $\sigma$ respectively
 satisfying that every member of $\mathcal I$ and $\mathcal J$ respectively
 remain as converging sequences, then $(x,y)$ is in the 
 closure of $H_f = \bigcup \{ D_n\cap \left( [0,f(n)]\times [0,f(n)]\right) : 
 n\in\omega\}$ with respect to the product topology given
 by $\hat\tau $ and $\hat \sigma $.
\end{theorem}

\begin{proof}
Let $x,y,\mathcal I$ and $\mathcal J$ be as described in the statement
of the Lemma.  If $x$ is isolated, then $\{x\}\times (\omega,\sigma)$ is M-separable,
and, in fact, there is a sequence $S\subset \{x\}\times \omega$ converging
to $(x,y)$ satisfying that $S$ is mod finite contained in 
 $\bigcup\{ D_n : n > m\}$ for all $m\in \omega$.
 Therefore we assume that neither $x$ nor $y$ are isolated and we
 choose infinite sequences $\langle x_n :n\in\omega\rangle$ converging
 to $x$ and $\langle y_n : n\in\omega\rangle$ converging to $y$.
 Fix pairwise disjoint families $\{U_n :n\in\omega\}$ 
 and $\{ W_n : n\in\omega\}$ of clopen sets in $\tau$ respectively $\sigma$
 so that, for all $n$, $x_n\in U_n$ and $y_n\in W_n$.

 Define the set $D = \bigcup \{ D_n \cap (U_n\times W_n) : n\in \omega\}$.
  Clearly $(x_n,y_n)$ is in the closure of $D$ for all $n\in\omega$.
  Assume there is some  $I\in \mathcal I$ and $J\in \mathcal J$
  satisfying that $(x,y)$ is in the closure of $D\cap (I\times J)$. 
  Since $\{(x,y)\}\cup \left(D\cap (I\times J)\right)$ is a metric
  space, it is M-separable, so again, there is sequence $S\subset D$
  that converges to $(x,y)$. By the choice of $D$, 
   $S\setminus \bigcup\{ D_m : m\geq n\}$ is finite for all $m$.

   So it remains to prove the Theorem in the case where
   $(x,y)$ is not in the closure of $D\cap (I\times J)$ for all $I\in 
   \mathcal I$ and $J\in \mathcal J$. Notice that this is equivalent
   to the case that for each such $I,J$ pair, we can remove a finite
   set from each and have that $D\cap (I\times J)$ is empty. 

  For each $f\in \omega^\omega$, let $H_f = \bigcup\{ D_n \cap (f(n)\times f(n)) : n\in\omega\}$
  as in the statement of the Theorem.
   Let $X$ be the following set
   $$ X = \{ (I,J,f) : I\in\mathcal I, J\in \mathcal J, f\in \omega^{\omega},
    D\cap (I\times J)=\emptyset\}~.$$
    We will identify $(I,J,f)$ from $X$ with the pair
     $(I(f), J(f))$ where
      $I(f) = H_f \cap (I\times \omega)$ and
       $J(f)= H_f\cap (\omega \times J)$. 
    For $(I,J,f)\in X$, $I(f)\cap J(f) \subset
        D\cap (I\times J)$ and so is empty.
We equip $X$ with the standard   topology on $(\mathcal P(\omega\times\omega))^2$
through this identification.     
The standard subbasic clopen subsets of $\mathcal P(\omega\times\omega)$
are sets of the form $\{ a\subset \omega\times\omega : (j,k)\in a\}$. 

Define the set $G\subset X^2\setminus \Delta_X$ by
the relation that $( ~ (I_1,J_1,f_1), (I_2,J_2,f_2)~)$ is in $G$
providing 
 $$I_1(f_1)\cap J_2(f_2)\neq\emptyset \ \ \mbox{or}\ \  
 I_2(f_2)\cap J_1(f_1)\neq\emptyset~~.$$ 
 It is trivial that $G$ is an open relation. 

 Suppose that $\{ Y_k : k\in \omega \} $ is a family of subsets
 of $X$ satisfying that $Y_k^2 $ is disjoint
 from $G$ for each $k\in\omega$. Assume towards a contradiction
 that $\bigcup_k Y_k = X$. 
   Let $X_0 $ be the set of pairs $(I,J)\in \mathcal I\times \mathcal J$
   that satisfy that $D\cap (I\times J)$ is empty. 
   For each $(I,J)\in X_0$, 
   since $\{(I,J)\}\times \omega^{\omega}$ is
   a subset of $\bigcup_k Y_k$,
there is a $k(I,J)\in\omega$ (minimal to be definite) so
   that $\{ f\in \omega^{\omega} : (I,J,f)\in Y_k\}$ 
   is $<^*$-cofinal in $\omega^\omega$. 
   Let $X_0[k] = \{ (I,J) : k(I,J) = k\}$ and observe
   that, since $Y_k^2\setminus \Delta_X$ is disjoint from $G$,
   it follows that
   for $(I_1,J_1),(I_2,J_2)\in X_0[k]$,
    $(I_1\cup I_2)\times (J_1\cup J_2)$ is disjoint
    from $D$. For each $k\in \omega$, set 
    $$A_k = \bigcup \{ I\in \mathcal I : ~(\exists J\in \mathcal J)~ (I,J)\in X_0[k]\}$$
    and
    $$B_k = \bigcup \{ J\in \mathcal J : ~(\exists I\in \mathcal I)~(I,J)\in X_0[k]\}~.$$
    Notice that $A_k\times B_k$ is disjoint from $D$ for all $k\in \omega$.
 
Fix any $n\in\omega$ and set $U_n^0=U_n$ and $W_n^0= W_n$. 
Clearly $(x_n,y_n)$ is a limit point of the interior
of $\overline{D\cap (U^0_n\times W^0_n)}$.
By recursion on 
 $k < n$, we define $U_n^{k{+}1}$ and $W_n^{k{+}1}$
 so that  
 \begin{enumerate}
\item $(x_n,y_n)$ is a limit point of the interior of
    $\overline{D\cap (U^{k{+}1}_n\times W^{k{+}1}_n)}$, and
    \item either $U^{k{+}1}_n =U^k_n\setminus A_k$, or
    \item $W^{k{+}1}_n = W^k_n\setminus B_k$.
 \end{enumerate}
Suppose we have chosen $U^k_n$ and $W^k_n$ and let
$S$ be the interior of $\overline{D\cap (U^k_n\times W^k_n)}$. 
By assumption $(x_n,y_n)$ is a limit point of $S$.
Work briefly in the subspace $S\cup \{(x_n,y_n)\}$.  The sets
$S\cap (D\cap ((U^k_n\cap A_k)\times W^k_n))$ and $S\cap (D\cap (U^k_n\times (W^k_n\cap  B_k)))$
are disjoint and so one of their interiors 
will not be dense in a neighborhood of $(x_n,y_n)$
 (in the subspace $S\cup \{ (x_n,y_n)\}$). 
 By
 symmetry assume this is so for $S\cap (D\cap ((U^k_n\cap A_k)\times W^k_n))$.
 Choose an open subset $S_0$ of $S$ so that $S_0\cup \{(x_n,y_n)\}$
 is open in $S\cup \{ (x_n,y_n)\}$ 
 and so that $S_0\cap (D\cap ((U^k_n\cap A_k)\times W^k_n))$ has empty interior.
 In this case we set $U^{k{+}1}_n = U^k_n\setminus A_k$
 and $W^{k{+}1}_n = W^k_n$, and notice that 
  $D\cap ((U^k_n\setminus A_k) \times W^k_n)$ is dense
  in $S_0$. 
 It follows that
  $S_0$ is contained in $\overline{D\cap (U^{k{+}1}_n\times W^{k{+}1}_n)}$. 
  Let $F_n $  be the set of $k< n$
  such that  $U^{n}_n$ is disjoint from $A_k$. Notice
  that  for $ k< n$ and $k\notin F_n$,
    $W^{n}_n$ is disjoint from $B_k$. 
     Let $s_n\in 2^n$ denote the characteristic function of $F_n$.

     Let $T\subset 2^{<\omega}$ denote the tree consisting
     of all $\{ s_n\restriction k : k \leq n \in \omega\}$
     and let $s\in 2^\omega$ be a branch of $T$. 
     For each $m\in \omega$, choose $n_m\in \omega$ so that
      $s_{n_m}\restriction m = s\restriction m$. Without loss
      of generality we can assume that the sequence
       $S= \{ n_m : m\in\omega\}$ is strictly increasing.
       For each $m\in\omega$, choose a sequence $I_{n_m}\subset 
        U^m_{n_m}$ that converges to $x_{n_m}$. 
        Since $x$ is in the closure of $\bigcup\{ I_{n_m} : m\in\omega\}$,
         we can choose a sequence $I\subset 
         \bigcup\{ I_{n_m} : m\in\omega\}$ that converges to $x$.
         Let $L = \{ m\in \omega : I\cap I_{n_m}\neq\emptyset$. 
         For each $m\in L$, choose $J_{n_m}\subset W^m_{n_m}$ that
         converges to $y_{n_m}$ and choose a sequence $J\subset
          \bigcup\{ J_{n_m}  : m\in L\}$ that converges to $y$.
          Let $L_2\subset L$ be the set of $m$ such
          that $J\cap J_{n_m}$ is not empty.

          Let us check that, for each $k\in \omega$, $(I,J)\notin X_0[k]$.
Indeed, consider any $k\in \omega$ and choose $          m\in L_2$
so that $k<n_m$.  If $k\in F_{n_m}$, then $(I,J)\notin X_0[k]$
because $I\cap A_k$ is empty.  If $k\notin F_{n_m}$, then $(I,j)\notin
 X_0[k]$ because $J\cap B_k$ is empty.
 \medskip

 At this stage of the proof, we have established, by 
 Lemma \ref{properOCA}, that there is a proper poset $P$ that forces
 there is a subset, which we will denote 
  $\{ (I_\alpha,J_\alpha, f_\alpha) : \alpha\in\omega_1\}\subset X$,
  satisfying that $((I_\alpha,J_\alpha,f_\alpha), (I_\beta,J_\beta,f_\beta))$
  is in $G$ for all $\alpha < \beta < \omega_1$. 
  Note that the family $\{ (I_\alpha(f_\alpha), J_\alpha(f_\alpha)) : \alpha
  <\omega_1\}$ is a Luzin gap (see Definition \ref{luzin}) and recall 
  that a Luzin gap remains a Luzin gap in any forcing extension that
  preserves $\omega_1$. 
  
  Now consider a
  $\sigma$-linked poset $Q$ that adds a dominating function $f_Q$ over
  the ground model. Suppose that $\mathcal F$ is a   filter
  on $Q$ that is generic over the ground model.  Assume,
  towards a contradiction, 
  that there is $Q$-name $\dot A$ for a subset of $D$
 that satisfies, for any $Q$-generic filter $\mathcal F$,
   $ I(f)  \subset^* \val_{\mathcal F}(\dot 
   A)$ and
    $J(f) \cap \val_{\mathcal F}(\dot A)$ is finite
   for all $(I,J,f)\in X$. This property of $\dot A$ will continue
    to hold in the forcing extension by $P$ since all
    the relevant dense open subsets of $Q$ will remain dense.
In addition, forcing by $P$ will preserve that $Q$ 
is $\sigma$-linked, and so the iteration $P * Q$ will also
be proper, and therefore preserve $\omega_1$.
However this is a contradiction, since the Luzin  gap
   $\{ (I_\alpha(f_\alpha),J_\alpha( f_\alpha)) : \alpha\in\omega_1\}$
 added by $P$  can not be split by the valuation
 of the $Q$-name $\dot A$.
    
  Now, with $f_Q$ being the dominating real added by $Q$,
   consider the set $H_{Q} = \bigcup \{ D_n\cap (f_Q(n)\times f_Q(n)) : n\in \omega\}$
   and note that, for all $f\in\omega^\omega$ in the ground model,
   $H_f \subset^* H_Q$.  Therefore, for all
  $(I,J,f)\in X$,   $I(f)\cup J(f)$ is mod finite contained in $H_Q$.   

   Assume that $\hat\tau $ and $\hat \sigma $ are topologies as in the statement
   of the Theorem. Assume that $x\in U\in \hat\tau $ and $y\in W\in\hat\sigma $.
   Let $(I,J,f)$ be any element of $X$.
   Clearly $A = H_Q\cap (U\times \omega)$ mod finite contains $I(f)$
   and $B = H_Q\cap (\omega\times W)$ mod finite contains $J(f)$.
   Since $A$ mod finite contains $I(f)$ for all 
    $(I,J,f)\in X$, it must meet $J(f)$ in an infinite
    set for some $(I,J,f)\in X$, so fix such an
    element $(I,J,f)$ of $X$. 
    Since $J(f)$ is mod finite contained in $B$ 
    and $A\cap B = H_Q\cap (U\times W)$, it follows
    that $H_Q\cap (U\times W)$ is infinite.
    
This completes the proof.           
\end{proof}

\section{Proof of Lemma \ref{tight}}

The goal of this section is to prove that if $\kappa^{<\kappa}=\kappa$,
then 
there is a ccc poset  $P_\kappa$
of cardinality $\kappa$ that produce a model of Martin's Axiom
in which there is a tight $\omega^\omega$-gap. We regret that
we have to prove this, but we are unable to find a suitable reference.
For a partial or total function $s$ from $\omega$ to $\omega$,
 let $s^{\downarrow} = \{ (m,j) : m\in\dom(s)\ \mbox{and}\ j< s(m)\}$.
 For functions $f,g\in \omega^\omega$, let $f\vee g$ denote
 the function $(f\vee g) (n)  = \max(f(n),g(n))$ for all $n\in\omega$.

Suppose that $\mathcal F = \{ f_\alpha : \alpha < \mathfrak b \}\subset \omega^\omega
$ is a mod finite increasing chain that is also dominating.
 Suppose that, for each $\alpha<\mathfrak b$, $h_\alpha$ is a 2-valued function
with domain $f_\alpha^\downarrow$. Say that $\mathcal H_\lambda
=\{ h_\alpha : \alpha < \lambda\}$ is coherent if, for all 
 $\beta < \alpha < \lambda$, the set of $(j,k)\in f_\alpha^\downarrow
 \cap f_\beta^\downarrow$ such that $h_\alpha( j,k)\neq h_\beta( j,k)$ is finite.
 Clearly it is necessary to prove there is such a coherent family
  $\mathcal H_{\mathfrak b}$ in the final model.

  \begin{definition}  Say that $\mathcal H_\lambda = \{ h_\alpha : \alpha < \lambda\}$ is a
   linear coherent family if it is a coherent family of 2-valued functions,
   such that
   for each $\alpha$, $\dom(h_\alpha) = f_\alpha^\downarrow$ for some $f_\alpha\in\omega^\omega$,
   and $\{ f_\alpha : \alpha < \lambda\}$ is  $\leq^*$-increasing.
  \end{definition}

\begin{definition}  If $\mathcal H_\lambda$ is a linear\label{defcoherent} coherent family, 
then $Q(\mathcal H_\lambda)$, also denoted   $Q(\{ h_\alpha : \alpha < \lambda\})$,
is the following poset. A condition $q\in Q(\mathcal H_\lambda)$, is
 a tuple $(s_q, h_q, F_q, f_q)$ satisfying 
 \begin{enumerate}
\item   $s_q\in \omega^{<\omega}$ with domain $n_q$ and $f_q\in \omega^\omega$,
\item $h_q $  is a 2-valued  function with domain  $s_q^\downarrow $,
 \item   $F_q$ is a finite subset of $\lambda$,
 \item for all $\alpha\in F_q$ and $\delta_q = \max(F_q)$,
 and for all $n_q\leq m$,
  $f_\alpha(m)\leq f_{\delta_q}(m) \leq f_q(m)$,
  and for all $(m,j)\in \dom(h_\alpha)\cap \dom(h_{\delta_q})$,
  $h_\alpha(m,j) = h_{\delta_q}(m,j)$.
 \end{enumerate}
 The ordering on $Q(\mathcal H_\lambda)$ is that $q\leq r$
 providing $s_q\supset r_q$, $h_q\supset h_r$, $F_q\supset F_r$,
  $f_q\geq f_r$, and for all $(m,j)\in \dom(h_q)\cap \dom(h_{\delta_r})$
  with $n_r\leq m$, $s_q(m,j) = h_{\delta_r}(m,j)$. 

  We let $\dot f$ and $\dot h$ be the two canonical $Q(\mathcal H_\lambda)$-names,
    (and in context we would  denote them as $\dot f_\lambda$ and $\dot h_\lambda$)
    where, if $G$ is a $Q(\mathcal H_\lambda)$-generic filter,
     $\val_{G}(\dot f) = \bigcup \{ s_q : q\in G\}$
     and $\val_{G}(\dot h) = \bigcup\{ h_q : q\in G\}$. 
\end{definition}

It should be clear that $\dot f$ is a dominating real added by
 $Q(\mathcal H_\lambda)$ (even if $\lambda = 0$)
and that, for each $\alpha\in \lambda$, 
 the set of $n\in\omega$ such that
  there are $f_\alpha(n)< j<k< f_\lambda(n)$
  with $h_\lambda(n,j)\neq h_\lambda(n,k)$ is
  cofinite.
Also, by the next proposition,  
 $\{ h_\alpha : \alpha < \lambda+1\}$
 is a linear coherent family extending $\mathcal H_\lambda$.

\begin{proposition} For\label{density}
each $\beta\in \lambda$,
 the set $D_\beta = \{ q\in Q(\{h_\alpha : \alpha < \lambda\})
 : \beta\in F_q\}$ is dense.  Furthermore, if
  $q\in Q(\{h_\alpha : \alpha<\lambda\})$ and 
  $ h_{\delta_q}\cup 
  h_\beta \restriction \left( [n_q,\omega)\times \omega)\right)
  $ is a function, then $(s_q,h_q, F_q\cup\{\beta\},f_q\vee f_\beta)$ 
  is an extension of $q$. 
\end{proposition}

  \begin{lemma}  For each $\delta\in \lambda$,
  the\label{centered} subset $S_\delta = \{ q\in Q(\{h_\alpha : \alpha \in \lambda\}): 
   \delta_q = \delta\}$ is $\sigma$-centered.
\end{lemma}

\begin{proof}     
 If
 $q,r\in S_\delta$ and $h_q=h_r$, then 
  $( s_q, h_q, F_q\cup F_r , f_q\vee f_r)$ is an extension
  of both $q$ and $r$ and is in $S_\delta$.
\end{proof}

 \begin{corollary}  If $\lambda$ has countable cofinality,
  then $Q(\{h_\alpha : \alpha\in\lambda\})$ is ccc 
  for every linear coherent sequence of length $\lambda$.
 \end{corollary}

 This next result is also a standard fact about gaps, but 
in a new setting.

\begin{corollary}  If $\{ h_\alpha : \alpha \leq  \lambda\}$ 
is a linear coherent gap, then\label{successor} 
 $Q(\{ h_\alpha : \alpha < \lambda\})$ is $\sigma$-centered. 
\end{corollary}

\begin{proof} By Lemmas \ref{centered} and \ref{density},
   $S_\lambda = \{ q\in Q(\{ h_\alpha : \alpha \leq \lambda\}):
    \lambda\in F_q\}$ 
   is dense and $\sigma$-centered in 
   $ Q(\{ h_\alpha : \alpha \leq \lambda\})$.
 Also, the    poset
$   Q(\{ h_\alpha : \alpha < \lambda\})$ is  subposet
of $Q(\{ h_\alpha : \alpha \leq \lambda\})$ 
and therefore is also $\sigma$-centered.
\end{proof}

 However, if $\lambda$ has cofinality $\omega_1$, 
  $Q(\mathcal H_\lambda)$ may not be ccc. 
    This next well-known  result is  due to Kunen
    (see \cite{Scheepers}) when applied to standard gaps.

  \begin{lemma} For  an uncountable\label{cccequivalent}
  linear coherent gap, $\{ h_\alpha : \alpha<\lambda\}$,
the poset   $  Q(\{ h_\alpha :\alpha < \lambda\})$ is ccc 
if and only if 
for every uncountable $X\subset\lambda$, there are
$\alpha< \beta$ in $X$ such that $h_\alpha\cup h_\beta$ is a function.
  \end{lemma}

\begin{proof}
Assume first that $Q(\{h_\alpha : \alpha < \lambda\})$ is ccc
and consider any uncountable $X\subset \lambda$. 
 By passing to a subset we may assume that $X$ has order-type
 $\omega_1$. Suppose first that $X$ has an upper bound
 $\mu<\lambda$. Then for each $\xi\in X$, there is an $n_\xi\in\omega$
 such that $h_\xi\restriction \left([n_\xi,\omega)\times\omega\right)\subset
  h_\mu$. Choose $\xi<\alpha$ both in $X$ so that
   $n = n_\xi = n_\alpha$ and $h_\xi\restriction n\times \omega$
   and $h_\alpha\restriction n\times\omega$ (which are both finite)
   are equal. Then $h_\xi\cup h_\alpha$ is a function.

Now suppose that $\lambda$ has cofinality $\omega_1$. In this
case we can force with $Q(\{h_\alpha : \alpha<\lambda\})$,
thus preserving $\omega_1$,
and repeat the argument in the previous paragraph using
 $\lambda=\mu$, i.e.  the new function $h_\lambda$ added
 by $Q(\{h_\alpha : \alpha<\lambda\})$.
 \medskip

  Now we prove the other direction and assume
  that for every uncountable $X\subset\lambda$,
  there are distinct $\alpha,\beta\in X$ satisfying
  that $h_\alpha\cup h_\beta$ is a function.
  Let $\{ q_\xi : \xi\in\omega_1\}$ be any subset
  of $Q(\{h_\alpha : \alpha\in\omega_1\})$. 
 By passing to a subcollection, we can 
 assume that there is a pair $s,h$ such
 that $(s,h) = (s_{q_\xi},h_{q_\xi})$ for
 all $\xi\in\omega_1$.
  Let $X  = \{ \delta_{q_\xi} : \xi\in\omega_1\}$
  and choose distinct $\xi,\eta\in\omega_1$ so
  that $h_{\delta_{q_\xi}}\cup h_{\delta_{q_\eta}}$ is a 
  function. It is easy to check
  that $(s, h, F_{q_\xi}\cup F_{q_\eta}, f_{q_\xi}\vee f_{q_\eta})$
  is a common extension of $q_\xi$ and $q_\eta$.
\end{proof}

The dominating real aspect of the linear coherent sequence poset
introduces some complications when utilized in an iteration
which we deal with by introducing
an alternate, but equivalent, formulation of the poset. We will
separate   each of the components $s$ and $h$ into two pieces where  one
piece is not allowed to be a  name. This is just to emphasize which
portion has been forced to have a specific value (or \textit{determined\/}
as it is often called).
We will abuse the standard
notation $\check{a}$  to mean that $\check{a}$ is a finite ground model set.

\begin{definition} For a poset $P$ and   $P$-names, 
 $\{ \dot f_\alpha, \dot h_\alpha : \alpha < \lambda\}$,
 that is forced to be a linear coherent sequence, we define\label{Q'} 
 the $P$-name $\dot Q'(\{ \dot h_\alpha : \alpha < \lambda\})$ as
 follows. A condition $q\in 
 \dot Q'(\{ \dot h_\alpha : \alpha < \lambda\})$ is a tuple
  $( \check{n}_q, \check{s}_q, \tau_q, \check{h}_q, \pi_q, \check{F_q}, \dot f_q)$
  where the following are forced by $1_{P}$: 
  \begin{enumerate}
  \item  $n_q\in \omega$, $s_q   \in \omega^{n_q}$, $s_q\leq \tau_q\in \omega^{n_q}$,
  \item $h_q$ is a 2-valued function with domain $s_q^{\downarrow}$,
  \item $h_q\subset \pi_q$ is a 2-valued function with domain $\tau_q^{\downarrow}$,
  \item $F_q$ is a finite subset of $\lambda$,
  \item for each $\alpha\in F_q$, $\dot f_\alpha \leq \dot f$.
  \end{enumerate}
Say that a condition $q\in \dot Q'(\{ \dot h_\alpha : \alpha < \lambda\})$ 
is pure if $\tau_q = \check{s}_q $ and $\pi_q = \check{h}_q$.

For each $q\in \dot Q'(\{\dot h_\alpha :\alpha < \lambda\})$, 
  let $\hat q = ( \tau_q , \pi_q, F_q, \dot f_q)$.  We note that
   $\hat q$ is forced to be an element of $Q(\{ \dot h_\alpha :\alpha < \lambda\})$
   and we define the ordering on 
$    \dot Q'(\{\dot h_\alpha :\alpha < \lambda\})$ 
by $q_1 \leq q_2$ if 
  $\hat{q}_1 \leq \hat{q}_2 $.
\medskip

 Suppose that $s\in \omega^n$ and $\tau$ is a $P$-name for an
 element of $\omega^n$ such that $1\Vdash s\leq \tau$. 
 For any 2-valued function $h$ with domain $s^{\downarrow}$,
  let $h\oplus 0_{\tau}$ denote the name of the function with domain
   $\tau^{\downarrow}$ that extends $h$ and has value $0$
   at all $(m,j)\in \tau^{\downarrow}\setminus s^\downarrow$.
\end{definition}

\begin{lemma}  Suppose that $\langle P_\alpha, \dot Q_\beta : \alpha \leq \lambda,
\beta <\lambda\rangle$ is a finitely\label{cohere} supported iteration,
and $\{ \dot f_\alpha , \dot h_\alpha : \alpha < \lambda \}$ are
 $P_\lambda$-names 
satisfying
\begin{enumerate}
\item for each $\alpha<\lambda$, each of $\dot f_\alpha$ and 
 $\dot h_\alpha$ are $P_{\alpha+1}$-names  satisfying that $\dot f_\alpha$
 is forced to be in $\omega^\omega$ and $\dot h_\alpha$ is a 2-valued
 function with domain $\dot f_\alpha^\downarrow$,
\item for each $\beta<\lambda$, $\dot Q_\beta$ is a $P_\beta$-name of a ccc poset,
\item for each even ordinal $\beta<\lambda$, if $P_\beta$ forces
that $\{ \dot h_\alpha : \alpha < \beta \}$ 
is a linear coherent family, then $\dot Q_\beta$ is 
the $P_\beta$-name  $\dot Q'(\{\dot h_\alpha : \alpha <\beta \})$,
and $\dot f_\beta, \dot h_\beta$ are the canonical $P_{\beta+1}$-names
 associated with $\dot Q_\beta$, otherwise $\dot Q_\beta = 2^{<\omega}$ and
 $\dot f_\beta=\dot f_0$,
 and $\dot h_\beta = \dot h_0$,
 \item if $\beta=\alpha+1<\lambda$ is an odd ordinal, then $\dot f_\beta=\dot f_\alpha$
 and $\dot h_\beta = \dot h_\alpha$.
 \end{enumerate}
Then for each $\beta \leq \lambda$, $P_\beta$ forces
that $\{ \dot h_\alpha : \alpha < \beta \}$ is
a linear coherent family. Furthermore, if $\beta$ has uncountable cofinality,
 then $P_\beta$ forces that for every uncountable $X\subset \beta$,
  there are distinct $\xi,\eta\in X$ such that
   $\dot h_\xi\cup \dot h_\eta$ is a function.
\end{lemma}

  \begin{proof}  It is immediate from the remarks immediately 
  after Definition \ref{defcoherent} that, for each $\beta <\lambda$,
   $P_\beta$ forces that $\{ \dot h_\alpha : \alpha < \beta\}$ is
    a linear coherent family. It, however, is not immediate
    that $P_\beta$ is ccc. We prove the second stated conclusion   
    of the Lemma by induction on $\beta \leq \lambda$. 
Since this conclusion is vacuous for $\beta < \omega_1$, we
might as well simply assume that it holds for all $\beta < \lambda$
and prove it for $\lambda$.  
It follows from Corollaries \ref{successor} and \ref{cccequivalent}
that for uncountable $X\subset \beta <\lambda$, it remains
true, i.e. not destroyed by further forcing,
that there are $\xi<\eta\in X$ such that
 $h_\xi\cup h_\eta$ is a function.
Needless to say, there is nothing to 
prove unless $\lambda$ has cofinality $\omega_1$. 

In preparation we make some observations, stated as Facts, about $P_\lambda$.

\begin{Fact} Let $ \tilde P_\lambda$ be the set of conditions that satisfy, for each
even $\beta\in \dom(p)$ and each odd $\alpha+1\in F_{p(\beta)}$,
 we also have that $\alpha$ is in $F_{p(\beta)}$.  Then
  $\tilde P_\lambda$ is a dense subset of $P_\lambda$.
\end{Fact}

\begin{Fact} Let   $p\in \tilde P_\lambda$ and suppose\label{fullfour}
that $\alpha <\beta$ are even ordinals in $\dom(p)$. Then $p$ will force
that $\dot h_\alpha\cup \dot h_\beta$ is a function if
\begin{enumerate}
\item $\alpha \in F_{p(\beta)}$,
\item $n_{p(\alpha)} = n_{p(\beta)}$, $s_{p(\alpha)} = s_{p(\beta)}$,
   $h_{p(\alpha)} = h_{p(\beta)}$,
   \item $\pi_{p(\alpha)} = h_{p(\alpha)} \oplus 0_{\tau_{p(\alpha)}}$ (as in Definition \ref{Q'}), and
   \item $\pi_{p(\beta)} = h_{p(\beta)} \oplus 0_{\tau_{p(\beta)}}$.
\end{enumerate}
\end{Fact}

\begin{Fact} Let   $p\in \tilde P_\lambda$ and suppose\label{addF}
that $\alpha <\beta$ are even ordinals in $\dom(p)$ and
that $\alpha\in F_{p(\beta)}$. Then $\bar p$ is an extension of $p$
where $\bar p(\gamma) = p(\gamma)$ for all $\beta\neq\gamma\in \dom(p)$
and 
 $\bar p(\beta)$ is equal to
 $
 ( n_{p(\beta)}, s_{p(\beta)}, \tau_{p(\beta)}, 
  h_{p(\beta)}, 
     \pi_{p(\beta)},   F_{p(\beta)}\cup F_{p(\alpha)}, \dot f_{p(\beta)})$.
\end{Fact}

\begin{Fact} Suppose that $p\in \tilde P_\lambda$ and that for
even $\alpha < \beta$ both in $\dom(p)$ we have\label{just3}
the conditions (2)-(4) of Fact \ref{fullfour} holding
and that $\delta_{p(\alpha)}$ is an element of $F_{p_\xi(\beta)}$.
Then the  condition $\bar p \leq p$ forces that $
\dot h_\alpha\cup \dot h_\beta$ is a function where
$\bar p$ is defined as follows: 
$\bar p(\gamma)=p(\gamma)$ for all $\beta\neq \gamma \in \dom(p)$,
and 
$$\bar p(\beta) = (n_{p(\beta)}, s_{p(\beta)},
 \tau_{p(\beta)}, h_{p(\beta)}, h_{p(\beta)}\oplus 0_{\tau_{p(\beta)}} , 
   F_{p(\beta)}\cup \{ \alpha\}, \dot f_{p(\beta)} \vee \dot f_\alpha)~.$$
\end{Fact}

\begin{Fact} Let $p\in \tilde P_\lambda$ and\label{addmu}
let $  \beta$ be an even ordinal
in $\dom(p)$ such that  $p(\beta)$ is a   pure condition
and let $\delta = \delta_{p(\beta)}$ be the maximum even ordinal 
in $  \max(F_{p(\beta)})$.  Choose
any $m\in\omega$ such that $n = n_{p(\beta)}< m$ and any
$\bar p\in P_{\delta+1}$ that forces   values $\bar s, \bar h$
on $\dot f_\delta\restriction [n,m)$ and $\dot h_\delta\restriction 
\bar s^{\downarrow}$ respectively.
Recall that $1_{P_\beta}$ forces that $\dot f_\delta\leq \dot f_{p(\beta)}$.
Then $\tilde p$ is an extension
of $\bar p$ and $p$ where $\tilde p\restriction \delta = \bar p$,
 $\tilde p(\gamma) = p(\gamma)$ for $\beta\neq \gamma \in \dom(p) 
 \setminus \delta{+}1$,
 and $\tilde p(\beta) =
 (m, s_{p(\beta)}\cup \bar s, \bar \tau , h_{p(\beta)}\cup \bar h,
 ( h_{p(\beta)}\cup \bar h)\oplus 0_{\bar\tau}, F_{p(\beta)}, \dot f_{p(\beta)})$
 where $\bar\tau = s_{p(\beta)}\cup \dot f_{p(\beta)}\restriction [n,m)$.
 \medskip

 Moreover, if for some even $\alpha <\beta$, $\bar p$ forces
 that $\dot h_\delta\cup (\dot h_\alpha \restriction \left([m,\omega)\times \omega )\right)$
 is a function, then we could instead define $\tilde p(\beta)$
 to equal $
 (m, s_{p(\beta)}\cup \bar s, \bar \tau , h_{p(\beta)}\cup \bar h,
 ( h_{p(\beta)}\cup \bar h)\oplus 0_{\bar\tau}, F_{p(\beta)}\cup\{\alpha\}, \dot f_{p(\beta)}
\vee \dot f_\alpha )$
\end{Fact}

\bigskip

With the benefit of the above Facts we are ready to prove Theorem \ref{cohere}.
Let $\dot X$
be a $P_\lambda$-name of an uncountable subset of $\lambda$.
By the definition of the family $\{ \dot h_\alpha : \alpha<\lambda\}$,
 we may assume that $\dot X$ is forced to consist of even ordinals.
Since we are proceeding by induction, we may assume
that $\dot X$ is forced to be cofinal in $\lambda$.
Let $e$ be a strictly increasing function from $\omega_1$
to a cofinal subset of $\lambda$.
 For each $\xi < \omega_1$, choose a condition $p_\xi\in \tilde P_\lambda$
 that forces some $ \beta_\xi \in \lambda\setminus e(\xi) $ is an element
 of $\dot X$. For each $\xi\in\omega_1$, let $\beta_\xi\in H_\xi$ denote the finite
 support of $p_\xi$. 

 For each $\xi\in\omega_1$, we make some additional assumptions
 about $p_\xi$. Let $\beta$ be the maximum even ordinal in $H_\xi$. 
 By possibly strengthening $p_\xi\restriction \beta$ we can ensure
 that $p_\xi\restriction \beta$ forces that $p_\xi(\beta)$ is pure. 
 We can also ensure that $F_{p_\xi(\beta)} \cap \alpha $ is a subset
 of $\dom(p_\xi)$, and  for each even $\alpha\in \dom(p_\xi)\cap \beta$,
 $F_{p_\xi(\beta )}\cap \alpha $ is a subset of $F_{p_\xi(\alpha)}$.   
 We can also ensure that $n_{p_\xi(\alpha)}\geq  n_{p_\xi(\beta)}  $ 
 for all even $\alpha\in \dom(p_\xi)$.
 This is step 1 of a finite recursion (since every descending
 sequence of ordinal is finite).  In this way we can 
 assume that, for each even ordinal $\beta$ in $\dom(p_\xi)$,
  $p_\xi\restriction \beta$ forces that $p_\xi(\beta)$ 
  is pure 
 and for even $\alpha\in H_\xi\cap \beta$,
 $F_{p_\xi(\beta)}\cap \alpha \subset
  F_{p_\xi(\alpha)}$.

 By passing to an uncountable subset we can 
 assume that each $H_\xi$ has cardinality $\ell$ and fix
 an increasing enumeration, $\{ \alpha(\xi,i) : i<\ell\}$ of $H_\xi$. 
 For each $i<\ell$ and $\xi,\eta<\omega_1$, we may assume
 that $\alpha(\xi,i)$ is even if and only if $\alpha(\eta,i)$ is even,
 and that  there is a fixed $\bar \imath<\ell$ so
 that $\beta_\xi  = \alpha(\xi, \bar \imath)$ for all $\xi$. 
 Let $E$ denote the set of $i<\ell$ such that (each)
  $\alpha(\xi,i)$ is even
We can also assume that for $\xi <\eta$ and $i\in E $,
$(n_i, s_i, h_i) =  
 (n_{p_\xi( \alpha(\xi,i))} , s_{p_\xi( \alpha(\xi,i))} , h_{p_\xi( \alpha(\xi,i))} ) = 
(n_{p_\xi( \alpha(\xi,i))} , s_{p_\eta(\alpha(\eta,   i))} , h_{p_\eta(\alpha(\eta,i))})
$.
 
 Notice that $\{ \alpha(\xi,\bar \imath) : \xi\in\omega_1\}$
 is unbounded in $\lambda$. 
 By a recursion of length at most $\bar \imath$, we can repeatedly
 pass to an uncountable subset of $\xi\in\omega_1$ so as to ensure,
  for each $i<\bar \imath$, either $\{ \alpha(\xi,i) : \xi\in\omega_1\}$
   is unbounded in $\lambda$ or has an upper bound $\mu_i<\lambda$.
   Let $i_0 \leq  \bar \imath$ be minimal so that 
    $\{ \alpha(\xi,i_0) : \xi\in\omega_1\}$ is unbounded. 
    Choose any even $\mu<\lambda$ so that $\alpha(\xi,i)<\mu$
    for all $\xi\in\omega_1$ and $i<i_0$. 
    For simple convenience assume that,
    for all $\xi\in\omega_1$, $\alpha(\xi,i_0)$ is an even ordinal. 
 Let   $\{ i_k : k <\bar\ell\} $ be an enumeration of
  $E\setminus i_0$.

By the inductive assumption, $P_{\mu+1}$ is ccc and so we may
choose a generic filter $G_{\mu+1}$ for $P_{\mu+1}$ such
that $\Gamma = \{ \xi \in\omega_1 : p_\xi \restriction \mu \in G_\mu\}$ is
uncountable. By re-indexing we may assume that $\mu+1< \alpha(\xi,i_0)$
for all $\xi\in\omega_1$, and by again choosing an
uncountable subsequence
 we can assume that, for $\xi<\eta$ both in $\Gamma$,
   $H_\xi \subset \alpha(\eta,i_0)$.

For each $\xi\in\Gamma$, let $\delta_\xi$ denote the maximum
element of $F_{p_\xi(\alpha(\xi,i_0))}$.  Since $\delta_\xi<\mu$
we may  
let $f_{\delta_\xi}$ and $h_{\delta_\xi}$ be the valuations
of $\dot f_{\delta_\xi}$ and  $\dot h_{\delta_\xi}$ respectively
by the filter $G_{\mu+1}$. Let
 $f_\mu$ and $h_\mu$ be defined analogously.
Now choose a value $\bar m\in\omega$ and an uncountable
 $\Gamma_1\subset \Gamma$ satisfying that
  $f_{\delta_\xi} \restriction [\bar m,\omega)\leq f_\mu$,
  and 
   $h_{\delta_\xi}\restriction [\bar m,\omega)\times\omega$
   is a subset of $h_\mu$ for all $\xi\in \Gamma_1$.

We work in the extension $V[G_{\mu+1}]$ and fix any $\xi\in \Gamma_1$.
We can ignore $p_\xi\restriction \mu$ since we have
that $p_\xi\restriction \mu\in G_{\mu+1}$.  
 For each even $\beta\in H_\xi\setminus \mu$, 
 we have that  $\delta_\xi$ is an
 element of $F_{p_\xi(\beta)}$ and is
 the maximum of $F_{p_\xi(\beta)}\cap \mu+1$.
 Let $\bar s = f_{\delta_\xi}\restriction [n_{i_0} , \bar m)$
 and $\bar h = h_{\delta_\xi}\restriction \bar s^\downarrow$. 
Applying the ``moreover'' clause of Fact \ref{addmu}, we have an extension
 $\bar p_\xi$ of $p_\xi$ such
 that $\bar p_\xi \restriction \mu{+}1\in G_{\mu+1}$ forces
 that
 $\bar s = \dot f_{\delta_\xi}\restriction [n_{i_0} , \bar m)$
 and $\bar h = \dot h_{\delta_\xi}\restriction \bar s^\downarrow$,
 $\bar p_\xi(\gamma) = p_\xi(\gamma)$ for $\alpha(\xi,i_0)< \gamma \in H_\xi$
 and
  $\bar p_\xi(\alpha(\xi,i_0))$ is as indicated in Fact \ref{addmu}
  with $\mu$ added to $ F_{\bar p_\xi(\alpha(\xi,i_0))}$.
  Next, by a finite recursion, we keep applying the first
  clause of Fact \ref{addmu}, so as to arrange that
    $\alpha(\xi,i_k)\in F_{\bar p_\xi(\alpha(\xi,i_{k+1}))}$
    and
     $n_{\bar p_\xi(\alpha(\xi,i_k))}=n_{\bar p_\xi(\alpha(\xi,i_{k+1}))}$
     for all
     $  k<\bar \ell-1$.  
     Actually we can stop at
      $\alpha(\xi,\bar\imath ) $ but nothing is saved.
      Additionally, by Fact \ref{addF}, we can assume
      that $F_{\bar p_\xi(\alpha(\xi,i_k))} \subset
         F_{\bar p_\xi(\alpha(\xi,i_{k+1}))}$ for all $k<\bar\ell$.

 Choose $\xi <\eta$ both in $\Gamma_1$
 so that, for all $k < \bar \ell-1$,
  $$( n_{\bar p_\xi(\alpha(\xi,i_k))} , 
 s_{\bar p_\xi(\alpha(\xi,i_k))} , 
  h_{\bar p_\xi(\alpha(\xi,i_k))})
  =( n_{\bar p_\eta(\alpha(\xi,i_k))} , 
 s_{\bar p_\eta(\alpha(\xi,i_k))} , 
  h_{\bar p_\eta(\alpha(\xi,i_k))})
  ~.$$
 Since $\bar p_\xi \restriction \mu+1$ and
 $\bar p_\eta\restriction \mu+1$ are both in $G_{\mu+1}$,
 and $H_\xi\cap H_\eta\subset \mu$, there is a condition
  $\bar p\in P_\lambda$ satisfying that
   $\bar p\setminus \mu+1 = (\bar p_\xi \restriction H_\xi\setminus(\mu+1))
   \cup (\bar p_\eta \restriction H_\eta \setminus (\mu+1))$.
 
  Recall that $\beta_\xi = \alpha(\xi, \bar\imath)$.
Note that $\mu$ is the largest element of $F_{\bar p_\eta(\alpha(\eta,i_0))}$
and that $\bar p_\xi$ forces that
     $ h_\mu\restriction [n_{\bar p_\xi(\beta_\xi)},\omega)
     \times\omega)$ is a subset of $\dot h_{\beta_\xi}$.     
     Therefore, by the \textit{moreover} clause of Fact \ref{addmu},
     we can extend $\bar p $ 
     to a condition $p'$ (but only in the coordinate
      $\alpha(\eta,i_0)$) as in Fact \ref{addmu} so
      that $F_{p'(\alpha(\eta,i_0))}$ is obtained
      by   adding  $\beta_\xi$ to $F_{\bar p_\eta(\alpha(\eta,i_0))}$,
      and, by Fact \ref{addF}, we can also arrange
      that $\beta_\xi $ is an element
      of $F_{  p' (\beta_\eta)}$. 
      The proof is finished by verifying
      that $\alpha  = \beta_\xi $ and $\beta = \beta_\eta$
      satisfy the conditions of Fact \ref{fullfour} for
      the condition $p'$.
         \end{proof}

Now we can complete the proof of Lemma \ref{tight}

\bgroup

\def\proofname{Proof of Lemma \ref{tight}}

\begin{proof} 
Let $\kappa$ be an uncountable regular cardinal satisfying $\kappa^{<\kappa}=\kappa$.
 Fix an enumeration $\{ R_\alpha : \alpha < \kappa\}$ of
   $H(\kappa)$ (the set of sets with transitive closure having cardinality
    less than $\kappa$, see \cite{BaumHandbk}).  

 Define the system $\langle P_\alpha, \dot Q_\beta : \alpha \leq \kappa,
\beta <\kappa\rangle$ as in  Lemma \ref{cohere} 
as well as the names
   $\{ \dot f_\alpha , \dot h_\alpha : \alpha < \lambda \}$, where,
   for each odd ordinal $\alpha<\kappa$,
    $\dot Q_\alpha$ is chosen so that if $R_{\gamma_\alpha} = \dot Q_\alpha$,
    then $\gamma_\alpha$ is the least ordinal $\gamma<\kappa$ satisfying
    that $R_{\gamma}$ is a $P_\alpha$-name of a ccc poset
    that is not an element of $\{ \dot Q_\eta : \eta < \alpha, \ \eta\
    \mbox{an odd ordinal}\}$. By Lemma \ref{cohere}, it follows
    that $\{ \dot f_\alpha , \dot h_\alpha : \alpha < \kappa\}$
    is forced to be a linear coherent family. By the definition
    of $\dot Q_\beta$ for even ordinals $\beta$, it is clear
    that $\{ \dot f_\alpha : \alpha < \kappa\}$ is a dominating
    family. We now prove that the $\omega^\omega$-gap
     $$\{ (\dot a_{f_\alpha} = \dot h_\alpha^{-1}(0),
      \dot b_{f_\alpha} = \dot h_\alpha^{-1}(1)) : \alpha <\kappa\}$$ 
    is a tight gap. Let $\dot X$ be any $P_\kappa$-name of 
    a subset of $\omega\times\omega$ such
    that it is forced that there is an
    infinite set of $n$ such
    that $\dot X\cap \left(\{n\}\times\omega\right)$
    is infinite.       
    Since, by Lemma \ref{cohere}, $P_\kappa$ is ccc, there
    is an even ordinal  $\lambda<\kappa$ such that $\dot A$ and $\dot X$
    are equivalent to $P_\lambda$-names. 
    Let $G_\mu$ be a $P_\mu$-generic filter and let $X$
    be the valuation of $\dot X$ by $G_\mu$.
    We prove that
    it will be forced that $\dot h_\lambda \restriction X$
    takes on values $0$ and $1$ infinitely often. 
    In $V[G_\mu]$, the valuation of the poset
     $\dot Q_\lambda$ is equivalent to the poset
      $Q = Q(\{ h_\alpha : \alpha < \lambda\})$. 
      To prove this, consider any condition $r\in Q$
      and simply note the trivial claim
      that there is an extension $q\in Q$
      satisfying that there is an $m>n_r$ and
       values $(m,i),(m,j)\in X$ such
       that $h_q((m,i))=0$ and $h_q(m,j)=1$.

       Finally we explain how our enumeration scheme ensured
       that Martin's Axiom holds in the forcing extension by $P_\kappa$.
    It suffices to prove that if $\dot Q\in H(\kappa)$ is a $P_\kappa$-name
    of a ccc poset and if  $\{\dot  D_\xi : \xi <\mu\}$, for some $\mu<\kappa$,
    is a set of $P_\kappa$-names for dense subsets of $\dot Q$,
    then there is a $P_\kappa$-name $\tilde G$ for a filter
    on $\dot Q$ that meets every $\dot D_\xi$. 
    Again, using that $P_\kappa$ is ccc and that $\mu$ and $|\dot Q|$
    are less than $\kappa$, there is a $\beta < \kappa$
    such that $\dot Q$ and every $\dot D_\xi$ is equivalent
    to $P_\beta$-names. Since we were lazy with our enumeration
    method we play a little trick. Choose any $\nu<\kappa$
    large enough so that the $P_\beta$-name for
    the iteration $\dot Q* \mbox{Fn}(\nu,2)$ is not
    in the list $\{ \dot Q_\alpha : \alpha < \beta\}$. 
    Let $\gamma<\kappa$ be such that $R_\gamma = \dot Q*\mbox{Fn}(\nu,2)$.
    Since $\dot Q$ is   ccc is the forcing extension by $P_\kappa$,
     it is, for every $\beta \leq \alpha < \kappa$, a $P_\alpha$-name
     of a ccc poset.
    By the definition of the iteration sequence,
     there is an odd ordinal $\alpha \geq  \beta$ satisfying 
     that $\gamma_\alpha =\gamma$.  It is a standard exercise
     that $P_{\alpha+1} = P_\alpha * \dot Q * \mbox{Fn}(\mu,2)$
     will add a filter on $\dot Q$ that meets
     every $\dot D_\xi$ (since these are all $P_\alpha$-names).
   \end{proof}

   \egroup

   \end{document}